\documentclass{amsart}
\usepackage{color}
\def\rojo{} \def\red{} \def\green{} \def\blue{}
\def\azuloso{}
\def\colorado{}

\usepackage[utf8]{inputenc}
\usepackage[utf8]{inputenc}
\usepackage[T1]{fontenc}
\usepackage[all]{xy}
\usepackage{lipsum}
\usepackage{url}
\usepackage{stackrel}

\usepackage{amsmath,amsthm,amssymb}
\usepackage{tikz}
\usepackage{graphicx}

\theoremstyle{definition}
\newtheorem{theorem}{Theorem}[section]
\newtheorem{definition}[theorem]{Definition}
\newtheorem{example}[theorem]{Example}
\newtheorem{proposition}[theorem]{Proposition}
\newtheorem{lemma}[theorem]{Lemma}
\newtheorem{remark}[theorem]{Remark}
\newtheorem{corollary}[theorem]{Corollary}

\numberwithin{equation}{section}

\begin{document}


\vspace{0.5in}

\renewcommand{\bf}{\bfseries}
\renewcommand{\sc}{\scshape}
\vspace{0.5in}

\title[Sectional category and The Fixed Point Property]%
{Sectional category and The Fixed Point Property \\ }

\author{Cesar A. Ipanaque Zapata}
\address{Departamento de Matem\'{a}tica,UNIVERSIDADE DE S\~{A}O PAULO
INSTITUTO DE CI\^{E}NCIAS MATEM\'{A}TICAS E DE COMPUTA\c{C}\~{A}O -
USP , Avenida Trabalhador S\~{a}o-carlense, 400 - Centro CEP:
13566-590 - S\~{a}o Carlos - SP, Brasil}
\curraddr{Departamento de Matem\'{a}ticas, CENTRO DE INVESTIGACI\'{O}N Y DE ESTUDIOS AVANZADOS DEL I. P. N.
Av. Instituto Polit\'{e}cnico Nacional n\'{u}mero 2508,
San Pedro Zacatenco, Mexico City 07000, M\'{e}xico}
\email{cesarzapata@usp.br}

\author{Jes\'{u}s Gonz\'{a}lez}
\address{Departamento de Matem\'{a}ticas, CENTRO DE INVESTIGACI\'{O}N Y DE ESTUDIOS AVANZADOS DEL I. P. N.
Av. Instituto Polit\'{e}cnico Nacional n\'{u}mero 2508,
San Pedro Zacatenco, Mexico City 07000, M\'{e}xico}
\email{jesus@math.cinvestav.mx}

\subjclass[2010]{Primary 55M20, 55R80, 55M30; Secondary 68T40}                                    %

\keywords{Fixed point property, Configuration spaces, Sectional category, Motion planning problem}
\thanks {The first author would like to thank grant\#2018/23678-6, S\~{a}o Paulo Research Foundation (FAPESP) for financial support.}

\begin{abstract} For a Hausdorff space $X$, we exhibit an unexpected connection between the sectional number of the Fadell-Neuwirth fibration $\pi_{\red{2},1}^X:F(X,2)\to X$, and the fixed point property (FPP) for self-maps on $X$. Explicitly, we demonstrate that a space $X$ has the FPP if and only if 2 is the minimal cardinality of open covers $\{U_i\}$ of $X$ such that each $U_i$ admits a continuous  local  section  for $\pi_{\red{2},1}^X$. This characterization connects a standard problem in fixed point theory to current research trends in topological robotics. 
\end{abstract}

\maketitle


\section{Introduction, \blue{outline and main results}}
\medskip\green{A topological theory of motion planning was initiated \rojo{in~\cite{farber2003topological}. As a result, Farber's topological complexity of the space of states of an autonomous agent and, more generally, the sectional number of a map} are numerical invariants \rojo{appearing naturally in the emerging field of topological robotics} (see \cite{pavesic} or \cite{pavesic2019}).}

\medskip\green{Let $X$ be a topological space and $k\geq 1$. The ordered configuration space of $k$ distinct points on $X$ (see \cite{fadell1962configuration}) is the topological space \[F(X,k)=\{(x_1,\ldots,x_k)\in X^k\mid ~~x_i\neq x_j\text{   whenever } i\neq j \},\] topologised as a subspace of the Cartesian power $X^k$. For $k\geq r\geq 1,$ there is a natural projection $\pi_{k,r}^X\colon F(X,k) \to F(X,r)$ given by $\pi_{k,r}^X(x_1,\ldots,x_r,\ldots,x_k)=(x_1,\ldots,x_r)$.}

\medskip \green{The study of sectional number and topological complexity for the map $\pi_{k,r}^X$ is still non-existent and, in fact, this work takes a first step in this direction. Several examples are presented to
illustrate \rojo{the} result \rojo{arising in this field.}}

\medskip \rojo{In more detail,} a topological space $X$ has  \textit{the fixed point property} (FPP) if, for every continuous self-map $f$ of $X$, there is a  point $x$ of $X$ such that $f(x)=x$. \blue{\rojo{We} address the natural question of whether (and how) the FPP can be characterized in the category of Hausdorff spaces and continuous maps.} Such characterizations are known in smaller, more restrictive categories. \blue{For} instance, Fadell proved in 1969 (see 
\cite{fadell1970} for references) that, in the category of connected compact metric ANRs:
\begin{itemize}
    \item If $X$ is a Wecken space, \blue{then} $X$ has the FPP if and only if $N(f)\neq 0$ for every self-map $f:X\to X$.
    \item If $X$ is a Wecken space satisfying the Jiang condition, $J(X)=\pi_1(X)$, then $X$ has the FPP if and only if $L(f)\neq 0$ for every self-map $f:X\to X$.
\end{itemize}

In this work we characterize the FPP \blue{within the category of Hausdorff spaces, and in terms of sectional number}. \green{Indeed, we demonstrate that a space $X$ has the FPP if and only if the sectional number $sec\hspace{.1mm}(\pi_{2,1}^X)$ \rojo{equals 2} (Theorem \ref{characterizacao-ppf}). \rojo{As a result,} we give an alternative proof of the fact that the real projective plane \rojo{has} the FPP (Example \ref{rp2}).}
 
\medskip
As shown in Section \ref{kr-robot}, a particularly interesting feature of our characterization comes from its connection to current research trends in topological robotics.

\medskip\green{On the other hand, the study of the Nielsen root number and the minimal root number for the map $\pi_{k,r}^X$ is still non-existent. This problem belongs to \rojo{the so-called \emph{unstable case} in} the general problem \rojo{of} coincidence theory (see \rojo{\cite[Section~7]{goncalves2005}}). \rojo{We} provide conditions in terms of the minimal root number of $\pi_{2,1}^X$ \rojo{for} $X$ \rojo{to} have the FPP (Proposition \rojo{\ref{conditions}}). \rojo{In addition, we} prove that the Nielsen root number $NR(\pi_{k,r}^X,a)$ is at most one (Proposition \ref{nrn-pi}).}

\medskip \green{The paper is organized as follows: In Section \ref{secrt}, we recall the notions of minimal root number $MR[f,a]$ and the Nielsen root number $NR(f,a)$. In Section \ref{sn}, we recall the notion of Schwarz genus, standard sectional number and basic results about these numerical invariants. \rojo{Our goal is to} study the sectional number for the projection map $\pi_{k,r}^X$. In particular, we demonstrate that a space $X$ has the FPP if and only if the sectional number $sec\hspace{.1mm}(\pi_{2,1}^X)$ \rojo{equals 2} (Theorem \ref{characterizacao-ppf}). In Section \ref{tc-map}, we recall the notion of topological complexity for a map and basic results about these numerical invariant. As applications of our results, in Section \ref{kr-robot}, we study a particular problem in robotics. }

\medskip\colorado{The authors of this paper deeply thank the referee for very valuable comments and timely corrections on previous versions of the work.}



\section{Root theory}\label{secrt}

In this section we give a brief exposition of standard mathematical topics in Root theory: the minimal root number and the Nielsen root number $NR(f,a)$. Our exposition is by no means complete, as we limit our attention to concepts that appear in geometrical and topological questions. More technical details can be found in standard works on root theory, like \cite{brooks1970number} or \cite{brown1999middle}.

Let $f:X\to Y$ be a continuous map between topological spaces, and fix $a\in Y$. A point $x\in X$ such that $f(x)=a$ is called a \textit{root} of $f$ at $a$. 

In Nielsen root theory, by analogy with Nielsen fixed-point theory, the roots of $f$ at $a$ are grouped into Nielsen classes, a notion of essentiality is defined, and the Nielsen root number $NR(f,a)$ is defined to be the number of essential root classes. The Nielsen root number is a homotopy invariant and measures the size of the root set in the sense that  $$NR(f,a)\leq MR[f,a]:=\min\{\mid g^{-1}(a)\mid\hspace{.2mm}\colon~~g\simeq f\}.$$ The number $MR[f,a]$ is called \textit{the minimal root number} for $f$ at $a$. A classical result of Wecken states that $NR(f,a)$ is in fact a
sharp lower bound in the homotopy class of $f$ for many spaces, in particular, for compact manifolds of dimension at least $3$. Thus, in \blue{such cases}, the vanishing of $NR(f,a)$ is sufficient to deform a map $f$ to be \blue{root-free}. Among the central problems in Nielsen root theory (or the theory of root classes) are:
\begin{itemize}
    \item the computation of $NR(f,a)$,
    \item the realization of $NR(f,a)$, i.e., deciding when \blue{the equality} $NR(f,a)=MR[f,a]$ holds.
\end{itemize}

\subsection{The Nielsen root number $NR(f,a)$} We recall from \cite{brooks1970number}
the Nielsen root number $NR(f,a)$. Let $f:X\to Y$ be a continuous map 
between path-connected topological spaces, 
and choose a point  $a\in Y$.

Assume that the set of roots $f^{-1}(a)$ is non empty. Two such roots $x_0$ and $x_1$ are \textit{equivalent} if there is a path $\alpha:[0,1]\to X$ from $x_0$ to $x_1$ such that the loop $f\circ\alpha$ represents the trivial element in $\pi_1(Y,a)$. This is indeed an equivalence relation, and an equivalence class is called a \textit{root class}.     

Suppose $H:X\times [0,1]\to Y$ is a homotopy. Then a root $x_0\in H_0^{-1}(a)$ is said to be \textit{$H$-related} to a root $x_1\in H_1^{-1}(a)$ if and only if there is a path $\alpha:[0,1]\to X$ from $x_0$ to $x_1$ such that the loop $\beta:[0,1]\to Y,~\beta(t)=H(\alpha(t),t)$ represents the trivial element in $\pi_1(Y,a)$.

Note that a root $x_0$ of $f:X\to Y$ is equivalent to another root $x_1$ if and only if $x_0$ is related to $x_1$ by the constant homotopy at $f$.

A root $x_0\in f^{-1}(a)$ is said to be \textit{essential} if and only if for any homotopy $H:X\times [0,1]\to Y$ beginning at $f$, there is a root $x_1\in H_1^{-1}(a)$ to which $x_0$ is $H$-related. If one root in a root class is essential, then all other roots in that root class are essential \blue{too}, and we say that the root class itself is \text{essential}. The number of essential root classes is called the \textit{Nielsen number} of $(f,a)$ and is denoted by $NR(f,a)$. The number $NR(f,a)$ is a lower bound for the number of solutions of $f(x)=a$. If $f^\prime$ is homotopic to $f$ then $NR(f,a)=NR(f^\prime,a)$. Furthermore, $NR(f,a)\leq MR[f,a]$.

The order of the cokernel of the fundamental group homomorphism $f_\#:\pi_1(X)\to \pi_1(Y)$ is denoted by $R(f)$, that is, \[R(f)=\left\lvert \dfrac{\pi_1(Y)}{f_\#(\pi_1(X))}\right\rvert;\] it depends only on the homotopy class of $f$. There are always at most $R(f)$ root classes of $f(x)=a$, in particular, $R(f)\geq NR(f,a)$. 

\begin{example}
If $f_\#:\pi_1(X)\to \pi_1(Y)$ is an epimorphism, $ NR(f,a)\leq 1$. In particular, if $Y$ is simply connected, then $ NR(f,a)\leq 1$.
\end{example}


\section{\red{Sectional number}}\label{sn}
\green{In this section we recall the notion of Schwarz genus \rojo{together with} basic results \rojo{from \cite{schwarz1958genus}} about this numerical invariant . Note that the notion of genus in Schwarz's paper \cite{schwarz1958genus} is \rojo{given} for a fibration. We shall follow the terminology in \cite{pavesic2019} and  refer to this notion as the Schwarz genus of a \colorado{continuous map}. Also, we recall from \cite{pavesic2019} the notion of standard sectional number.}

Let $p:E\to B$ be a \colorado{continuous map}.  A \textit{(homotopy) cross-section} or \textit{section} of $p$ is a (homotopy) right inverse of $p$, i.e., a map $s:B\to E$, such that $p\circ s = 1_B$ ($p\circ s \simeq 1_B$). Moreover, given a subspace $A\subset B$, a \textit{(homotopy) local section} of $p$ over $A$ is a (homotopy) section of the restriction map $p_|:p^{-1}(A)\to A$, i.e., a map $s:A\to E$, such that $p\circ s$ is (homotopic to) the inclusion $A\hookrightarrow B$.

We recall the following definitions.
\begin{definition}
\begin{enumerate}
    \item The (standard) \textit{sectional number} of a \colorado{continuous map} $p\colon E\to B$, $sec\hspace{.1mm}(p)$, is the minimal \blue{cardinality of} open \blue{covers} of $B$, such that each element \blue{of the cover} admits  a  continuous  local  section to $p$. 
    \item The \textit{sectional category} of $p$, also called Schwarz genus of $p$, and denoted by $secat(p),$ is the minimal \blue{cardinality of open covers of $B$, such that each element of the cover admits  a  continuous  homotopy local  section to $p$.}
\end{enumerate}
\end{definition}

\azuloso{Note that $p$ is surjective whenever $sec\hspace{.1mm}(p)<\infty$. The corresponding assertion for $secat(p)$ may fail.}

\begin{remark}\label{secat-sec}
We have $secat(p)\leq sec\hspace{.1mm}(p)$. Furthermore, if $p$ is a fibration then $sec\hspace{.1mm}(p) = secat(p)$.
\end{remark}

\begin{lemma}\label{prop-sectional-category}\cite{schwarz1958genus}
Let $p:E\to B$ be a continuous map and $R$ be a commutative ring with unit. If there exist cohomology classes $\alpha_1,\ldots,\alpha_k\in H^\ast(B;R)$ with $p^\ast(\alpha_1)=\cdots=p^\ast(\alpha_k)=0$ and $\alpha_1\cup\cdots\cup \alpha_k\neq 0$, then $sec\hspace{.1mm}(p)\geq k+1$.
\end{lemma}

\rojo{A few observations worth keeping in mind in what follows are:}
\begin{itemize}
\item \rojo{If} $B$ is path-connected (this case will appear in our work), we have that $\alpha\in H^\ast(B;R)$, $\alpha\neq 0$ with $p^\ast(\alpha)=0$ implies $\alpha\in \widetilde{H}^\ast(B;R)$.
\item \rojo{If} $p:E\to B$ \rojo{is} a continuous map \rojo{and} $p_\ast:H_\ast(E;R)\to H_\ast(B;R)$ or $p_\#:\pi_\ast(E)\to \pi_\ast(B)$ are not su\red{r}jective  then $sec\hspace{.1mm}(p)\geq 2$.
\item Let $p:E\to B$ be a continuous map. If $p$ has a section $s:B\to E$, then $p\circ s=1_B$ and $s^\ast\circ p^\ast=1_{H^\ast(B;R)}$. In particular, $p^\ast:H^\ast(B;R)\to H^\ast(E;R)$ is a monomorphism.
\end{itemize}

The following statement is well-known.

\begin{lemma}\label{pullback}\cite{schwarz1958genus}
Let $p:E\to B$ be a continuous map. If the following square
\begin{eqnarray*}
\xymatrix{ E^\prime \ar[r]^{\,\,} \ar[d]_{p^\prime} & E \ar[d]^{p} & \\
       B^\prime  \ar[r]_{\,\, f} &  B &}
\end{eqnarray*}
is a pullback. Then $sec\hspace{.1mm}(p^\prime)\leq sec\hspace{.1mm}(p)$.
\end{lemma}

\colorado{We recall the pathspace construction from \cite[pg. 407]{hatcheralgebraic}. For a continuous map $f:X\to Y$, consider the space 
\begin{equation*}
E_f=\{(x,\gamma)\in X\times PY\mid~\gamma(0)=f(x)\}.
\end{equation*} The map \begin{equation*}
\rho_f:E_f\to Y,~(x,\gamma)\mapsto \rho_f(x,\gamma)=\gamma(1),
\end{equation*} is a fibration. \azuloso{Further,} the projection over the first coordinate $E_f\to X,~(x,\gamma)\mapsto x$ is a homotopy equivalence with homotopy inverse $c:X\to E_f$ given by $x\mapsto (x,\gamma_{f(x)})$, where $\gamma_{f(x)}$ is the constant path at $f(x)$. \azuloso{This factors} an arbitrary map $f:X\to Y$ as the composition $X\stackrel{c}{\to} E_f\stackrel{\rho_f}{\to} Y$ of a homotopy equivalence and a fibration.}

\azuloso{For convenience, we record the following standard properties:}
\begin{proposition}\label{secat-pf-equal-secat-f}
\begin{enumerate}
    \item \azuloso{For a continuous map $f:X\to Y$, 
${secat}(\rho_f)= {secat}(f).$}
\item \colorado{If $f\simeq g$, then \azuloso{${secat}(f)={secat}(g).$}}
\end{enumerate}
\end{proposition}

Next, we recall the notion of LS category which, in our setting, is one greater than that given in \cite{cornea2003lusternik}. For example, the category of a contractible space is one.

\begin{definition}
The \textit{Lusternik-Schnirelmann category} (LS category) or category of a topological space $X$, denoted cat$(X)$, is the least integer $m$ such that $X$ can be covered by $m$ open sets, all of which are contractible within $X$. 
\end{definition}

We have $\text{cat}(X)=1$ iff $X$ is contractible. The LS category is a homotopy invariant, i.e., if $X$ is homotopy equivalent to $Y$ (which we shall denote by $X\simeq Y$), then $\text{cat}(X)=\text{cat}(Y)$.  

\medskip \colorado{The following lemma \azuloso{generalizes} Proposition 9.14 from \cite{cornea2003lusternik}.}

 \begin{lemma}\label{prop-secat-map}
 \colorado{Let $p:E\to B$ be a continuous map.
 \begin{enumerate}
     \item If $p$ is a fibration, then $sec\hspace{.1mm}(p)\leq \azuloso{\text{cat}}(B)$. In particular, for any continuous map $f:X\to Y$, we have $secat(f)\leq \text{cat}(Y)$.
     \item If $p$ is nulhomotopic, then $secat(p)=\text{cat}(B).$
 \end{enumerate}}
\end{lemma}
\begin{proof}
 \colorado{The first part of item $(1)$ was proved in \cite[Proposition 9.14]{cornea2003lusternik}. For the second part of item $2$, by Proposition~\ref{secat-pf-equal-secat-f}, we have $secat(f)=secat(\rho_f)$ and thus $secat(f)\leq \text{cat}(Y)$.}
 
\colorado{Item $2$ follows easily, because the inequality $\text{cat}(B)\leq secat(\rho_p)$ holds for any nulhomotopic map $p:E\to B$.} 
\end{proof}

\subsection{Configuration spaces}\label{secconfespa}

Let $X$ be a topological space and $k\geq 1$. The \textit{ordered configuration space} of $k$ distinct points on $X$ (see \cite{fadell1962configuration}) is the topological space \[F(X,k)=\{(x_1,\ldots,x_k)\in X^k\mid ~~x_i\neq x_j\text{   whenever } i\neq j \},\] topologised as a subspace of the Cartesian power $X^k$.

For $k\geq r\geq 1,$ there is a natural projection \blue{$\pi_{k,r}^X\colon F(X,k) \to F(X,r)$ given by $\pi_{k,r}^X(x_1,\ldots,x_r,\ldots,x_k)=(x_1,\ldots,x_r)$.}


\begin{lemma}[Fadell-Neuwirth fibration \cite{fadell1962configuration}] \label{TFN} Let $M$ be a connected $m-$dimensional topological manifold without boundary, where $m\geq 2$. \blue{For $k> r\geq 1$, the map} $\pi_{k,r}^M:F(M,k)\to F(M,r)$  is a locally trivial bundle with fiber $F(M-Q_r, k-r)$, \red{where $Q_r\subset M$ is a finite subset with $r$ elements}. In particular, $\pi_{k,r}^M$ is a fibration.
\end{lemma}

\red{\blue{The boundary restriction in Lemma~\ref{TFN} is important, for} $\pi_{k,r}^X:F(\blue{M},k)\to F(\blue{}M,r)$ \rojo{might fail to be} a fibration if $\blue{M}$ is a manifold with boundary. \blue{This can be seen} by considering, for example, the manifold $\mathbb{D}^2$, with $k=2$ and $r=1$, \blue{for} the fibre $\mathbb{D}^2-\{(0,0)\}$ is not homotopy equivalent to the fibre $\mathbb{D}^2-\{(1,0)\}$.}

\begin{proposition}\label{nrn-pi}
Let $M$ be a connected $m-$dimensional topological manifold without boundary, where $m\geq 2$. \blue{For $k> r\geq 1$, the projection} $\pi_{k,r}^M:F(M,k)\to F(M,r)$ has Nielsen root number $NR(\pi_{k,r}^M,a)\leq 1$ for any $a\in F(M,r)$. 
\end{proposition}
\begin{proof}
The map $\pi_{k,r}^M:F(M,k)\to F(M,r)$ is a fibration with fiber $F(M-Q_r, k-r)$. We note that $F(M-Q_r, k-r)$ is path-connected. By the long exact homotopy sequence of the fibration $\pi_{k,r}^M$, we have the induced homomorphism $(\pi_{k,r}^M)_\#:\pi_1F(M,k)\to \pi_1F(M,r)$ is an epimorphism. Then, $R(\pi_{k,r}^M)=1$ and thus the Nielsen root number $NR(\pi_{k,r}^M,a)\leq 1$ for any $a\in F(M,r)$. 
\end{proof}

\rojo{Note that $MR[\pi_{k,1}^X,a]=0$ (in particular $NR(\pi_{k,1}^X,a)=0$) for any contractible space $X$.}

\begin{proposition}\label{secop-pi-k-1}[Key lemma]
 For any $k\geq 2$ and $X$ a Hausdorff space, we have
$sec\hspace{.1mm}(\pi_{k,1}^X)\leq k$.
\end{proposition}
\begin{proof}
Fix an element $(p_1,\ldots,p_k)\in F(X,k)$. For each $i=1,\ldots,k$, set \[U_i:=X-\{p_1,\ldots,p_{i-1},p_{i+1},\ldots,p_k\}\] and \blue{let} $s_i:U_i\longrightarrow F(X,k)$ \blue{be} given by $s_i(x):=(x,p_1,\ldots,p_{i-1},p_{i+1},\ldots,p_k)$. We note that each $U_i$ is open (because $X$ is Hausdorff) and each $s_i$ is a local section of $\pi_{k,1}^X$. Furthermore,, $X=U_1\cup\cdots\cup U_k$. Thus, $sec\hspace{.1mm}(\pi_{k,1}^X)\leq k.$
\end{proof}

\begin{remark}
\colorado{\azuloso{Using} Lemma~\ref{prop-secat-map} \azuloso{we see that,} for any $k\geq2$,
\begin{equation}\label{dosconds}
\mbox{$\pi^X_{k,1}\simeq\star\;\;$ and $\;\;secat\hspace{.1mm}(\pi_{k,1}^X)=1\;\;$ if and only if $\;\;X\simeq\star$.}
\end{equation}
The most appealing situation of \azuloso{(\ref{dosconds})} holds for $k=2$, as in fact $sec\hspace{.1mm}(\pi_{2,1}^X)\in\{1,2\}$, in view of Proposition~\ref{secop-pi-k-1}. Indeed, it would be interesting to know whether there is a space $X$ for which $\pi^X_{2,1}$ is a nulhomotopic fibration having $sec\hspace{.1mm}(\pi_{2,1}^X)=2$. Such a space would have to be a non-contractible co-H-space of topological complexity 2 or 3 (see \azuloso{Proposition~\ref{nul-homotopy-implie-cat2}} and Remark~\ref{aaaaa}) and, more relevantly for the purposes of this paper, would have to satisfy the fixed point property ---see Definition~\ref{defifpp} and Theorem~\ref{characterizacao-ppf} below.}
\end{remark}

\begin{definition}\label{defifpp}
A topological space $X$ has  \textit{the fixed point property} (FPP) if for every continuous self-map $f$ of $X$ there is a  point $x$ of $X$ such that $f(x)=x$.
\end{definition}

\begin{example}
It is well known that the unit disc $D^m=\{x\in\mathbb{R}^m:~\parallel x\parallel\leq 1\}$ has the FPP (The Brouwer's  fixed point theorem). The \red{real, complex and quaternionic} projective spaces, $\mathbb{RP}^{n}, \mathbb{CP}^{n}$ and $\mathbb{HP}^{n}$ have the FPP \red{when $n$ is even} (see \cite{hatcheralgebraic}). For the particular case, $\mathbb{RP}^{2}$, see Example \ref{rp2}.
\end{example}

Note that the map $\pi_{2,1}^X:F(X,2)\to X$ admits a cross-section if and only if there exists a fixed point free self-map $f:X\to X$. Thus, we have the following theorem.
\begin{theorem}\label{characterizacao-ppf}[\blue{Main} theorem]
Let $X$ be a Hausdorff space. The space $X$ has the FPP if and only if $sec\hspace{.1mm}(\pi_{2,1}^X)=2$.
\end{theorem}
\begin{proof}
Suppose \blue{that} $X$ has the FPP, then $sec\hspace{.1mm}(\pi^{\red{X}}_{2,1})\geq 2$ \blue{so,} by the Key Lemma (Proposition \ref{secop-pi-k-1}), $sec\hspace{.1mm}(\pi_{2,1}^X)=2$.  Suppose \blue{now that} $sec\hspace{.1mm}(\pi_{2,1}^X)=2$, \blue{so in} particular $sec\hspace{.1mm}(\pi_{2,1}^X)\neq 1$. Hence, $X$ has the FPP. 
\end{proof}

\begin{example}
No nontrivial topological group $G$ has the FPP. Indeed, the map $s:G\to F(G,2),~g\mapsto (g,g_1g)$ (for some fixed $g_1\neq e\in G$) is a cross-section for $\pi_{2,1}^G:F(G,2)\to G$. The self-map $G\to G,~g\mapsto g_1g$ is fixed point free.
\end{example}

\begin{example}
We recall that the odd-dimensional projective spaces $\mathbb{RP}^{2n+1}$ has not the FPP, because there is a continuous self-map $h:\mathbb{RP}^{2n+1}\to \mathbb{RP}^{2n+1},$ given by the formula $h([x_1:y_1:\cdots:x_{n+1}:y_{n+1}])=[-y_1:x_1:\cdots:-y_{n+1}:x_{n+1}]$, without fixed point. Thus, $sec\hspace{.1mm}(\pi_{2,1}^{\mathbb{RP}^{2n+1}})=1.$

On the other hand, we know that an even-dimensional projective spaces $\mathbb{RP}^{2n}$ has the FPP. Thus, $sec\hspace{.1mm}(\pi_{2,1}^{\mathbb{RP}^{2n}})=2$. Analogous facts hold for complex and quaternionic projective spaces.
\end{example}

\begin{example}
The spheres $S^n$ does not have the FPP, because the antipodal map $A:S^n\to S^n,~x\mapsto -x$ has not fixed points. Thus, $sec\hspace{.1mm}(\pi_{2,1}^{S^n})=1.$
\end{example}

\begin{example}
We know that any closed surface $\Sigma$, except for the projective plane $\Sigma\neq\mathbb{RP}^2$, has not the FPP. Thus, $sec\hspace{.1mm}(\pi_{2,1}^{\Sigma})=1.$
\end{example}

\begin{corollary}\label{suficiente-ppf}
Let $X$ be a Hausdorff space. If there exist $\alpha\in H^\ast(X;R)$ with $\alpha\neq 0$ and $(\pi_{2,1}^X)^\ast(\alpha)=0\in H^\ast(F(X,2);R)$, that is, if the induced homomorphism $(\pi_{2,1}^X)^\ast:H^\ast(X;R)\to H^\ast(F(X,2);R)$ is not injective, then $sec\hspace{.1mm}(\pi_{2,1}^X)= 2$.  In particular, $X$ has the FPP.
\end{corollary}
\begin{proof}
From Lemma \ref{prop-sectional-category}, $sec\hspace{.1mm}(\pi_{2,1}^X)\geq 1+1=2$. Then, by Proposition \ref{secop-pi-k-1}, $sec\hspace{.1mm}(\pi_{2,1}^X)= 2$. Thus, the result follows from Theorem \ref{characterizacao-ppf}.
\end{proof}

\rojo{The converse of Corollary \ref{suficiente-ppf} is not true. For example, we recall that the unit disc $D^{m}:=\{x\in \mathbb{R}^m\mid~\parallel x\parallel\leq 1\}$ has the FPP (from the Brouwer`s fixed point theorem) and thus $sec\hspace{.1mm}(\pi_{2,1}^{D^{m}})=2.$ However,  $\widetilde{H}^\ast(D^{m};R)=0$.}

\begin{corollary}
Let $X$ be a Hausdorff space. If the induced homomorphisms $(\pi_{2,1}^X)_\ast:H_\ast(F(X,2);R)\to H_\ast(X;R)$ or $(\pi_{2,1}^X)_\#:\pi_\ast(F(X,2))\to \pi_\ast(X)$ are not surjective, then $sec\hspace{.1mm}(\pi_{2,1}^X)= 2$. In particular, $X$ has the FPP.
\end{corollary}

\begin{example}\label{rp2}
It is easy to see that $\pi_2(F(\mathbb{RP}^2,2))=0$ is trivial and $\pi_2(\mathbb{RP}^2)=\mathbb{Z}$. Then the induced homomorphism $(\pi_{2,1}^{\mathbb{RP}^2})_\#:\pi_2(F(\mathbb{RP}^2,2))\to \pi_2(\mathbb{RP}^2)$ is not surjective, and thus $sec\hspace{.1mm}(\pi_{2,1}^{\mathbb{RP}^2})= 2$.  In particular, $\mathbb{RP}^2$ has the FPP. This part can also be proved by employing Lefschetz`s fixed point theorem.
\end{example}

\begin{remark}
For $k\geq l\geq r$, consider the following diagram
\begin{eqnarray*}
 \xymatrix{ F(X,k) \ar[r]^{\pi_{k,l}^X}\ar[d]_{\pi_{k,r}^X} & F(X,l)\ar[dl]^{\pi_{l,r}^X} \\
            F(X,r) }
\end{eqnarray*}
It is easy to see that if $\pi_{l,r}^X\simeq$ \red{const}, then $\pi_{k,r}^X\simeq$ \red{const} for any $k\geq l\geq r$. Moreover, we have $MR[\pi_{l,r}^X,a]\geq MR[\pi_{k,r}^X,a] \text{ for any } k\geq l\geq r$. 
\end{remark}

\begin{proposition}\label{conditions}
\colorado{Let $X$ be a connected CW complex with $MR(\pi_{2,1}^X,x_0)=0$. Assume that there exist  $\alpha\in \widetilde{H}^\ast(X;R)$ with $\alpha\neq 0$ and $i^\ast(\alpha)=0\in \widetilde{H}^\ast(X-\{x_0\};R)$ for some $x_0\in X$, that is, $i^\ast:\widetilde{H}^\ast(X;R)\to \widetilde{H}^\ast(X-\{x_0\};R)$ is not injective, where $i:X-\{x_0\}\hookrightarrow X$ is the inclusion map.}
Then $sec\hspace{.1mm}(\pi_{2,1}^X)=2$. In particular, $X$ has the FPP.
\end{proposition}
\begin{proof}
%
From $MR(\pi_{2,1}^X,x_0)=0$, there exist a continuous map $\varphi:F(X,2)\to X$ such that $\varphi^{-1}(x_0)=\emptyset$ and $\varphi\simeq\pi_{2,1}^X$. We have the \blue{homotopy} commutative diagram
\begin{eqnarray*}
 \xymatrix{ F(X,2) \ar[r]^{\pi_{2,1}^X}\ar[d]_{\varphi} & X\\
            X-\{x_0\}.\ar[ur]^{i}  }
\end{eqnarray*}
The fact $\pi_{2,1}^X\simeq i\circ\varphi$ implies $\varphi^\ast\circ i^\ast=(\pi_{2,1}^X)^\ast$. In particular, $(\pi_{2,1}^X)^\ast(\alpha)=\varphi^\ast\circ i^\ast(\alpha)=0$. Therefore, there exist $\alpha\in \widetilde{H}^\ast(X;R)$ with $\alpha\neq 0$ and $(\pi_{2,1}^X)^\ast(\alpha)=0\in \widetilde{H}^\ast(F(X,2);R)$, then $sec\hspace{.1mm}(\pi_{2,1}^X)=2$.
\end{proof}

\begin{example}
    For $\pi_{2,1}^{S^2\vee S^1}:F(S^2\vee S^1,2)\to S^2\vee S^1$, we have $MR[\pi_{2,1}^{S^2\vee S^1},x_0]\geq 1$ for any $x_0\in S^2\vee S^1$. Indeed, we
\colorado{show below that} $sec\hspace{.1mm}(\pi_{2,1}^{S^2\vee S^1})=1$. Also, there exist $\alpha\in \widetilde{H}^1(S^2\vee S^1;R)$ with $\alpha\neq 0$ and $i^\ast(\alpha)=0\in \widetilde{H}^1(S^2;R)$,  
\colorado{so Proposition~\ref{conditions} yields} $MR[\pi_{2,1}^{S^2\vee S^1},x_0]\neq 0$, \colorado{as asserted. Now, in order to construct a cross-section for $\pi_{2,1}^{S^2\vee S^1}$, it suffices to exhibit a selfmap $f\colon S^2\vee S^1\to S^2\vee S^1$ with no fixed points. Think of $S^2\vee S^1$ as $\left(S^2\times \{b_0\}\right)\cup \left(\{a_0\}\times S^1\right)$, where $a_0=(1,0,0)$ and $b_0=(1,0)$. Then the required map $f$ is} 
given by the formulae
    \begin{align*}
     \colorado{f(a,b_0)} &
     = \colorado{(a_0,\gamma(a_1)), \mbox{ for any $a=(a_1,a_2,a_3)\in S^2$, and}}\\
     \colorado{f(a_0,b)} &=\colorado{(a_0,-b), \text{ for any } b\in S^1,}
    \end{align*}
\colorado{where $\gamma\colon [-1,1]$ is a path in $S^1$ from $b_0$ to $-b_0$.}
\end{example}

We next relate our results to Farber's topological complexity, a homotopy invariant of $X$ introduced in \cite{farber2003topological}. Let $PX$ denote the space of all continuous paths $\gamma: [0,1] \to X$ in $X$ and  $e_{0,1}: PX \to X \times X$ denote the map associating to any path $\gamma\in PX$ the pair of its initial and end points, i.e., $e_{0,1}(\gamma)=(\gamma(0),\gamma(1))$. Equip the path space $PX$ with the compact-open topology. 

\begin{definition}\cite{farber2003topological}
The \textit{topological complexity} of a path-connected space $X$, denoted by TC$(X)$, is the least integer $m$ such that the cartesian product $X\times X$ can be covered with $m$ open subsets $U_i$ such that, for any $i = 1, 2, \ldots, m$, there exists a continuous local section $s_i:U_i \to PX$ of $e_{0,1}$, that is, $e_{0,1}\circ s_i = id$ over $U_i$. If no such $m$ exists, we set TC$(X)=\infty$. 
\end{definition}

We have $\text{TC}(X)=1$ if and only if $X$ is contractible. The TC is a homotopy invariant, i.e., if $X\simeq Y$ then $\text{TC}(X)=\text{TC}(Y)$. Moreover, $\text{cat}(X)\leq \text{TC}(X)\leq 2\text{cat}(X)-1$ for any path-connected CW complex $X$. 

\begin{proposition}\label{nul-homotopy-implie-cat2}
Let $X$ be a non-contractible path-connected \colorado{CW complex}. If $\pi_{2,1}^X\simeq x_0$ for some $x_0\in X$ then $X-\{x_0\}$ is contractible in $X$. Furthermore $\text{cat}(X)=2$, \colorado{$\text{TC}(X)\in\{2,3\}$} and \colorado{$sec(\pi_{2,1}^X)=secat(\pi_{2,1}^X)=2$}. 
\end{proposition}
\begin{proof}
 Let $H:F(X,2)\times [0,1]\to X$ be a homotopy between $\pi_{2,1}^X$ and $x_0$. Set $G:(X-\{x_0\})\times [0,1]\to X$ given by the formula $G(x,t)=H((x,x_0),t)$. We have $G(x,0)=x$ and $G(x,1)=x_0$ for any $x\in X-\{x_0\}$. Thus $X-\{x_0\}$ is contractible in $X$. \rojo{This obviously yields cat$(X)=2$, as well as $2\leq\text{TC}(X)\leq3$. \colorado{Furthermore, we have $2\geq sec(\pi_{2,1}^X)\geq secat(\pi_{2,1}^X)=\text{cat}(X)=2$.} 
 }
\end{proof}


\begin{remark}\label{aaaaa}
It is well known that $\text{cat}(X)=2$ corresponds to the case in which $X$ is a co-H-space. This is a large class of spaces \rojo{including} all suspensions. In addition there are well-known examples of co-H-spaces that are not suspensions. In particular, a potential space satisfying the hypothesis in Proposition~\ref{nul-homotopy-implie-cat2} must be a co-H-space \colorado{of topological complexity 2 or 3 and would have to satisfy the fixed point property, but cannot be a closed smooth manifold. This last condition follows from the positive solution to the topological Poincaré conjecture.}
\end{remark}

\begin{remark}\cite{fadell1962configuration}
Let $X$ be a topological space. The map $\pi_{k,1}^X:F(X,k)\longrightarrow X$ has a continuous section, i.e., $sec\hspace{.1mm}(\pi_{k,1}^X)=1$ if and only if there exist $k-1$ fixed point free continuous self-maps $f_2,\ldots,f_{k}:X\longrightarrow X$ which are non-coincident, that is, $f_i(x)\neq f_j(x)$ for any $i\neq j$ and $x\in X$.
\end{remark}

\begin{example}
Let $G$ be a topological group with \blue{cardinality} $|G|\geq k$. Then $sec\hspace{.1mm}(\pi_{k,1}^G)=1$, because the map $s:G\to F(G,2),~g\mapsto (g,g_1g,\ldots,g_{k-1}g)$ is a cross section for $\pi_{k,1}^G$ (for some fixed $(g_1,\ldots,g_{k-1})\in F(G-\{e\},k-1)$).
\end{example}

\begin{example}\cite{fadell1962configuration}
Let $M$ be a topological manifold without boundary and let $Q_m\subset M$ be a finite subset with $m$ elements. Then $sec\hspace{.1mm}(\pi_{k,1}^{M-Q_m})=1$ for any $m\geq 1$.
\end{example}

\begin{proposition}\label{secop-pi-k-r}
\blue{Let $X$ be a Hausdorff space.} For any $k>r\geq 1$, we have 
$sec\hspace{.1mm}(\pi_{k,r}^X)\leq \binom{k}{r}$, 
where $\binom{k}{r}=\dfrac{k!}{r!(k-r)!}$, 
the standard binomial coefficient.
\end{proposition}
\begin{proof}
Let $(p_1,\ldots,p_k)\in F(X,k)$ be a fixed $k-$tuple. Set $Q_k:=\{p_1,\ldots,p_k\}$ and for each $I_r\subseteq Q_k$ with $| I|=r$, let $Q_{I_r}:=Q_k-I_r=\{p_{j_1},\ldots,p_{j_{k-r}}\}$, where $j_1<\cdots<j_{k-r}$. Set $U_{I_r}:=F(M-Q_{I_r},r)$, and \blue{let} $s_{I_r}:U_{I_r}\to F(X,k)$ \blue{be} given by $s_{I_r}(x_1,\ldots,x_r):=(x_1,\ldots,x_r,p_{j_1},\ldots,p_{j_{k-r}})$. We note $U_{I_r}$ is open in $F(M,r)$ and each $s_I$ is a local section of $\pi_{k,r}^X$. Furthermore, $F(M,r)=\bigcup_{I\subseteq Q_k,~\mid I\mid=r}U_I$. Then, $sec\hspace{.1mm}(\pi_{k,r}^X)\leq \binom{k}{r}$.
\end{proof}

\begin{corollary}\label{general-case}
Let $M$ be a connected topological manifold without boundary of dimension at least two. Then  $sec\hspace{.1mm}(\pi_{k,r}^M)\leq \min\{\binom{k}{r},cat(F(M,r))\}$.
\end{corollary}

\begin{example}\label{sec-contractible}
Let $M$ be a contractible topological manifold without boundary of dimension at least two. Then $sec\hspace{.1mm}(\pi_{k,1}^M)\leq \min\{k,cat(M)=1\}=1$. In particular, $M$ does not have the FPP.
\end{example}

\begin{remark}\label{diagram-spheres}
\blue{From} the diagram
\begin{eqnarray*}
 \xymatrix{ F(X,k) \ar[r]^{\pi_{k,k-1}^X}\ar[d]_{\pi_{k,1}^X} & F(X,k-1)\ar[dl]^{\pi_{k-1,1}^X} \\
            X }
\end{eqnarray*}
\blue{it is clear} that, if $\pi_{k,1}^X$ admits a section, then $\pi_{k-1,1}^X$ also admits a section. Thus $sec\hspace{.1mm}(\pi_{k,1}^X)=1$ implies $sec\hspace{.1mm}(\pi_{r,1}^x)=1$ for any $r\leq k$. Furthermore, when $X=S^d$, \blue{the corresponding} diagram
\begin{eqnarray*}
 \xymatrix{ F(S^d,k) \ar[r]^{\pi_{k,2}^{S^d}}\ar[d]_{\pi_{k,1}^{S^d}} & F(S^d,2)\ar[dl]^{\pi_{2,1}^{S^d}} \\
            S^d }
\end{eqnarray*}
and \blue{the fact} that $\pi_{2,1}^{S^d}$ always admits a section \blue{imply that,} if  $\pi_{k,2}^{S^d}$ admits a section, \blue{then}, so does $\pi_{k,1}^{S^d}$. The converse is also true, i.e., if  $\pi_{k,1}^{S^d}$ admits a section, \blue{then} so does $\pi_{k,2}^{S^d}$ \cite{fadell1962configuration}.
\end{remark}

\begin{proposition}\label{sec-even-spheres}
\red{If} $k>2$ and $d$ even, \red{then} $sec\hspace{.1mm}(\pi_{k,r}^{S^d})=cat(F(S^d,r))=2$, for \blue{$r\in\{1,2\}$.}
\end{proposition}
\begin{proof}
First, we show that $sec\hspace{.1mm}(\pi_{k,r}^{S^d})\geq 2$ for any $d$ even, $k\geq 3$ and $r=1$ or $2$. 
\red{By} the above diagrams, it suffices to show that $sec\hspace{.1mm}(\pi_{3,1}^{S^d})\geq 2$ for $d$ even, that is, $\pi_{3,1}^{S^d}$ does not admit a cross section. If a cross section existed, it would generate a map $f:S^d\to S^d$ such that $f(x)\neq x$ and $f(x)\neq -x$ for any $x\in S^d$. \red{Indeed, suppose that $\pi_{3,1}^{S^d}$ admits a section, it implies that $\pi_{3,2}^{S^d}$ admits a section (see the last part of Remark \ref{diagram-spheres}), say $s:F(S^d,2)\to F(S^d,3)$. Recall that a section $\sigma$ to $\pi_{2,1}^{S^d}$ is given by the formulae $\sigma(x)=(x,-x)$ for any $x\in S^d$. Take $f=p_{3}\circ s\circ\sigma$, where $p_3$ is the projection on the third coordinate.}  Since $f(x)\neq -x$  for every $x\in S^d$ it is easy to see that $f\simeq 1$ and $f$ has degree one and hence fixed points which is a contradiction (We recall that if $f:S^d\to S^d$ has not fixed points then $f$ is homotopic to the antipodal map and $f$ has degree $(-1)^{d+1}$). Thus, $\pi_{3,1}^{S^d}$ does not admit a cross section.

From Proposition \ref{general-case}, $sec\hspace{.1mm}(\pi_{k,r}^{S^d})\leq \min\{\binom{k}{r},cat(F(S^d,r))=cat(S^d)=2\}$. Then $sec\hspace{.1mm}(\pi_{k,r}^{S^d})=2 \text{ for any } k\geq 3, r\in \{1,2\}$ ($d$ even).
\end{proof}


\begin{proposition}\label{prop-prop}
\begin{enumerate}
    \item \red{If $L$ is a deformation retract of $X$}, then $\red{secat}(\pi_{k,1}^L)\geq \red{secat}(\pi_{k,1}^X)$.
    \item \cite{fadell1962configuration} If $M$ is \red{a smooth manifold} and admits a non-vanishing vector field, then $sec\hspace{.1mm}(\pi_{k,1}^M)=1$ for every $k$.
\end{enumerate}
\end{proposition}
\begin{proof}
 $(1)$ Let $r:X\to L$ be a deformation retraction, i.e., $r\circ i=1_L$ and $i\circ r\simeq 1_X$, where $i:L\to X$ is the inclusion map. We have the following commutative diagram
\begin{eqnarray*}
 \xymatrix{ F(L,k) \ar[r]^{i^k}\ar[d]_{\pi_{k,1}^L} & F(X,k)\ar[d]^{\pi_{k,1}^X} \\
           L \ar[r]^{i} & X }
\end{eqnarray*} 
Suppose $U\subset L$ is an open set of $L$ with \red{homotopy} local section $s:U\to F(L,k)$ of $\pi_{k,1}^L$. Set $V=r^{-1}(U)\subset X$ and consider $\sigma:V\to F(X,k)$ given by $\sigma=i^k\circ s\circ r$.
\[
\xymatrix{ r^{-1}(U) \ar[r]^{r}\ar@{->}@/_20pt/[rrr]_{\sigma} & U \ar[r]^{s} & F(L,k) \ar[r]^{i^k} & F(X,k)}
\]

We have that $\sigma$ is a \red{homotopy} local section of $\pi_{k,1}^X$. Therefore, $\red{secat}(\pi_{k,1}^L)\geq \red{secat}(\pi_{k,1}^X)$.  \end{proof}

\red{From (\cite{fadell1962configuration}, Theorem $5$-$(b)$) we have if $L\subset X$ is a retract and $\pi_{k,1}^L$ admits a cross-section then $\pi_{k,1}^X$ admits a cross-section. The statement from Proposition \ref{prop-prop} does not hold when $L$ is a retract. For example, $X=S^d$ (with $d\geq 2$ even) and $L=S^d_{-}=\{(x_1,\ldots,x_{d+1})\in S^d:~~x_{d+1}\leq 0\}$. Note that $S^d_{-}$ is a retraction of $S^d$, the retraction map is given by $r:S^d\to S^d_{-}$, $r(x)=x$ if $x\in S^d_{-}$ and $r(x)=(x_1,\ldots,x_d,-x_{d+1})$ if $x_{d+1}\geq 0$. We have $S^d_{-}$ is contractible, indeed it is homeomorphic to the $d-$dimensional closed unit disc $\mathbb{D}^d$ and thus $secat(\pi_{k,1}^L)=1$. Here we consider $k>2$ and thus $secat(\pi_{k,1}^X)=2$ (see Proposition \ref{sec-even-spheres} and Remark \ref{secat-sec}).}

\begin{corollary}\cite{fadell1962configuration}
If $M$ is compact and the first Betti number of $M$ does not vanish, then $sec\hspace{.1mm}(\pi_{k,1}^M)=1$ for every $k$.
\end{corollary}

\begin{corollary}\cite{fadell1962configuration}\label{sec-odd-manifolds}
If $M$ is an odd-dimensional differentiable manifold, \blue{then} $sec\hspace{.1mm}(\pi_{k,1}^M)=1$ for every $k$.
\end{corollary}

\begin{corollary}\label{sec-odd-spheres}
\red{If} $k>2$ and $d$ odd, \red{then} $sec\hspace{.1mm}(\pi_{k,r}:F(S^d,k)\to F(S^d,r))=1$ for \blue{$r\in\{1,2\}$.}
\end{corollary}
\begin{proof}
\red{\blue{This} follows from Corollary \ref{sec-odd-manifolds} and the last part of Remark \ref{diagram-spheres}.}
\end{proof}


\section{Topological complexity of a map}\label{tc-map}

Recall that $PE$ denotes the space of all continuous paths $\gamma: [0,1] \longrightarrow E$ in $E$ and  $e_{0,1}: PE \longrightarrow E \times E$ denotes the map associating to any path $\gamma\in PE$ the pair of its initial and end points $\pi(\gamma)=(\gamma(0),\gamma(1))$. Equip the path space $PE$ with the compact-open topology. 

Let $p:E\to B$ be a \colorado{continuous map}, and let $e_p:PE\to E\times B,~e_p=(1\times p)\circ e_{0,1}$.

\begin{definition} 
The \textit{topological complexity} of the map $p$, denoted by TC$(p)$, is the sectional number $sec\hspace{.1mm}(e_p)$ of the map $e_p$, that is, the least integer $m$ such that the cartesian product $E\times B$ can be covered with $m$ open subsets $U_i$ such that for any $i = 1, 2, \ldots , m$ there exists a continuous local section $s_i : U_i \longrightarrow PE$ of $e_{P}$, that is, $e_{P}\circ s_i = id$ over $U_i$. If no such $m$ exists we set TC$(p)=\infty$. 
\end{definition}

We use a definition of topological complexity which generally is not the same that given in \cite{pavesic2019}. However, under certain conditions, these two definitions coincides (see \cite{pavesic2019}). 

The proof of the following statement proceeds by analogy with \cite{pavesic2019}.

\begin{proposition}
For a map $p:E\to B$, we have $\text{TC}(p)\geq\max\{ \text{cat}(B),sec\hspace{.1mm}(p)\}$.
\end{proposition}
\begin{proof}
Let $U\subset E\times B$ be an open subset and $s:U\to PE$ be a partial section of $e_p$. Fix $x_0\in E$ and consider the inclusion $i_0:B\to E\times B$, given as $i_0(b)=(x_0,b)$. Set $V=i_0^{-1}(U)\subset B$, it is an open subset of $B$. 
Consider the map $H:V\times [0,1]\to B$ given by $H(b,t)=p(s(x_0,b)(t))$. It is easy to check that $H$ is a null-homotopy. We conclude that $\text{TC}(p)\geq cat(B)$.

On the other hand, consider the map $\sigma:V\to E$ defined by $\sigma(b)=s(x_0,b)(1)$. One can easily see that $\sigma$ is a partial section over $V$ to $p$. Therefore, $\text{TC}(p)\geq sec\hspace{.1mm}(p)$.
\end{proof}

The proof of the following statement proceeds by analogy with \cite{pavesic2019}.

\begin{proposition}\label{section-ineq}
Consider the diagram of maps $E^\prime\stackrel{p^\prime}{\to} E\stackrel{p}{\to} B\stackrel{p^{\prime\prime}}{\to}B^\prime$. If $p$ admits a section, then
 \begin{itemize}
     \item[$a$)] $\text{TC}(p^{\prime\prime})\leq \text{TC}(p^{\prime\prime} p)$.
     \item[$b$)] $\text{TC}(p p^\prime)\leq \text{TC}(p^\prime)$.
 \end{itemize}
  In particular, $\text{TC}(B)\leq\text{TC}(p)\leq \text{TC}(E)$. 
\end{proposition}
\begin{proof}
Let $s:B\to E$ be a section to $p$. 

$a)$ Suppose $\alpha_{p^{\prime\prime} p}:U\to PE$ is a partial section of $e_{p^{\prime\prime} p}$ over $U\subset E\times B^\prime$. Set $V:=(s\times 1_{B^\prime})^{-1}(U)\subset B\times B^\prime$. Then we can define the continuous map $\alpha_{p^{\prime\prime}}:V\to PB$ by \[\alpha_{p^{\prime\prime}}(b,b^\prime)(t):=\begin{cases}
    b, & \hbox{for $0\leq t\leq \frac{1}{2}$;} \\
    p(\alpha_{p^{\prime\prime} p}(s(b),b^\prime)(2t-1)), & \hbox{for $\frac{1}{2}\leq t\leq 1$.}
\end{cases}\] Since $\alpha_{p^{\prime\prime}}$ is a partial section of $e_{p^{\prime\prime}}$ over $V$, we conclude that $\text{TC}(p^{\prime\prime})\leq \text{TC}(p^{\prime\prime} p)$.

\smallskip
$b)$ Let $\alpha_{p^\prime}:U\to PE^\prime$ be a partial section to $e_{p^\prime}:PE^\prime\to E^\prime\times E$ over $U\subset E^\prime\times E$. Set $V:=(1_{E^\prime}\times s)^{-1}(U)\subset E^\prime\times B$ and define the continuous map $\alpha_{pp^\prime}:V\to PE^\prime$ given by $\alpha_{pp^\prime}(e^\prime,b):=\alpha_{p^\prime}(e^\prime,s(b))$. It follows that $\alpha_{pp^\prime}$ is a partial section of $e_{pp^\prime}$ over $V$. This implies $\text{TC}(p p^\prime)\leq \text{TC}(p^\prime)$.  
\end{proof}


\begin{theorem}\label{tc-implies-fpp}
Let $X$ be a Hausdorff space.
\begin{enumerate}
    \item If $X$ has the FPP, then $\text{TC}(\pi_{k,1}^X)\geq \max\{ \text{cat}(X),2\}$ for any $k\geq 2$.
    \item If $\text{TC}(\pi_{2,1}^X)<\text{TC}(X)$ or $\text{TC}(\pi_{2,1}^X)> \text{TC}(F(X,2))$, then $sec\hspace{.1mm}(\pi_{2,1}^X)=2$. In particular, $X$ has the FPP.
    \item If $X$ is a non-contractible space which does not have the FPP, then the configuration space $F(X,2)$ is not contractible.
\end{enumerate}
\end{theorem}
\begin{proof}
\rojo{$(1)$:} We have $\text{TC}(\pi_{k,1})\geq sec\hspace{.1mm}(\pi_{k,1})\geq  2$. We recall that, $sec\hspace{.1mm}(\pi_{2,1})=2$ implies $sec\hspace{.1mm}(\pi_{k,1})\geq 2$, for any $k\geq 2$.

\rojo{$(2)$:} This follows from Proposition \ref{section-ineq}.

\rojo{$(3)$:} By Proposition \ref{section-ineq}, we have $1<\text{TC}(X)\leq \text{TC}(\pi_{2,1})\leq \text{TC}(F(X,2))$ and thus $F(X,2)$ is not contractible.
\end{proof}

\rojo{Item (3) in Theorem \ref{tc-implies-fpp}} gives a partial generalization of the work \rojo{in~\cite{zapata2017non}.}

\begin{example}
We know that the unit disc $D^m$ has the FPP. Then $\text{TC}(\pi_{k,1}^{D^m})\geq 2$, for any $k\geq 2$.
\end{example}

The following Lemma generalizes the statement given in (\cite{pavesic2019}, pg. 19). 

\begin{lemma}\label{general-pullback}
If $p:E\to B$ is a fibration and $p^\prime:B\to B^\prime$ is a continuous map, then the following diagram is a pullback \begin{eqnarray*}
\xymatrix{ PE \ar[r]^{\,\,p_{\#}} \ar[d]_{e_{p^\prime p}} & PB \ar[d]^{e_{p^\prime}} & \\
       E\times B^\prime  \ar[r]_{\,\, p\times 1_{B^\prime}} &  B\times B^\prime &}
\end{eqnarray*}
\end{lemma}
\begin{proof}
 For any $\beta:X\to PB$ and any $\alpha:X\to E\times B^\prime$ satisfying $e_{p^\prime}\circ\beta=(p\times 1_{B^\prime})\circ\alpha$, we will check that there exists $H:X\to PE$ such that
$e_{p^\prime\circ p}\circ H=\alpha$ and $p_{\#}\circ H=\beta$.
\begin{eqnarray*}
\xymatrix{ X \ar@/^10pt/[drr]^{\,\,\beta} \ar@/_10pt/[ddr]_{\alpha} \ar@{-->}[dr]_{H} &   &  &\\
& PE \ar[r]^{\,\,p_{\#}} \ar[d]^{e_{p^\prime\circ p}} & PB \ar[d]^{e_{p^\prime}} & \\
       & E\times B^\prime   \ar[r]_{\quad p\times 1_{B^\prime}\quad} &  B\times B^{\prime} &}
\end{eqnarray*} Indeed, note that we have the following commutative diagram:
\begin{eqnarray*}
\xymatrix{ X \ar[r]^{\,\,p_{1}\circ\alpha} \ar[d]_{i_0} & E \ar[d]^{p} \\
       X\times I  \ar[r]_{\,\,\beta} &  B}
\end{eqnarray*} where $p_1$ is the projection onto the first coordinate. Because $p$ is a fibration, there exists $H:X\times I\to E$ satisfying $H\circ i_0=p_1\circ\alpha$ and $p\circ H=\beta$, thus we does.
\end{proof}

The following statement was proved in \cite{pavesic2019}; we give an \blue{elementary} proof in our context. 

\begin{proposition}
If $p:E\to B$ is a fibration, then $\text{TC}(p^\prime p)\leq \text{TC}(p^\prime)$ for any $p^\prime:B\to B^\prime$. In particular, $\text{TC}(p)\leq \text{TC}(B)$. 
\end{proposition}
\begin{proof}
 Since $p:E\to B$ is a fibration, the following diagram is a pullback (see Lemma \ref{general-pullback}) \begin{eqnarray*}
\xymatrix{ PE \ar[r]^{\,\,p_{\#}} \ar[d]_{e_{p^\prime p}} & PB \ar[d]^{e_{p^\prime}} & \\
       E\times B^\prime  \ar[r]_{\,\, p\times 1_{B^\prime}} &  B\times B^\prime &}
\end{eqnarray*} This implies $\text{TC}(p^\prime p)=sec\hspace{.1mm}(e_{p^\prime p})\leq sec\hspace{.1mm}(e_{p^\prime})= \text{TC}(p^\prime)$. 
 \end{proof}
 
 \begin{corollary}
 If $p:E\to B$ is a fibration that admits a section, then $\text{TC}(p)=\text{TC}(B)$. In particular, $\text{TC}(p)=1$ if and only if $B$ is contractible.
 \end{corollary}
 

\section{The $(k,r)$ robot motion planning problem}\label{kr-robot}

In this section we use the results above within a particular problem in robotics.

Recall that, in general terms, the \textit{configuration space} or \textit{state space} of a system $\mathcal{S}$ is defined as the space of all possible states of $\mathcal{S}$ (see \cite{latombe2012robot} or \cite{lavalle2006planning}). Investigation of the problem of simultaneous collision-free motion planning for a multi-robot system consisting of $k$ distinguishable robots, each with state space $X$, leads us to study the ordered configuration space $F(X,k)$ of $k$ distinct points on $X$.  \red{Recall the definition of \blue{the} ordered configuration \blue{space}} $F(X,k)$ in Subsection~\ref{secconfespa}. Note that the \red{$i$-th} coordinate of a point $(x_1,\ldots,x_n)\in F(X,k)$ represents the configuration of the \red{$i$-th} moving object, so that the condition $x_i\neq x_j$ reflects the collision-free requirement.

\textit{The $(k,r)$ robot motion planning problem} consists in controlling simultaneously these $k$ robots without collisions, where one is interested in the initial positions of the $k$ robots and \textit{only interested in the final position of the first $r$ robots ($k\geq r$)} (see Figure \ref{fig1}).
\begin{figure}[htb]
    \centering
\begin{tikzpicture}
\filldraw [gray] (-2,5) circle (0pt) node[anchor=west] {$X$};
\filldraw [gray] (0,0) circle (2pt) node[anchor=west] {$a_1$} node[anchor=east] {$(1)$};
\filldraw [gray] (1,3) circle (2pt) node[anchor=west] {$a_2$} node[anchor=east] {$(2)$};
\filldraw [gray] (2,4) circle (2pt)  node[anchor=west] {$b_1$} node[anchor=east] {$(1)$};
 \end{tikzpicture}
    \caption{The $(2,1)$ robot motion planning problem: we need to move Robots $1$ and $2$, simultaneously and avoiding collisions, from the initial positions $(a_1,a_2)$ to a final position $b_1$ of Robot $1$. We are only interested in the final position of the first robot.}
    \label{fig1}
\end{figure}

\textit{An algorithm} for the $(k,r)$ robot motion planning problem is a function which assigns to any pair of configurations $(A,B)\in F(X,k)\times F(X,r)$ consisting of an initial state $A=(a_1,\ldots,a_k)\in F(X,k)$ and a desired state $B=(b_1,\ldots,b_r)\in F(X,r)$, a continuous motion of the system starting at the initial state $A$ and ending at the desired state $B$ (see Figure \ref{fig2}). 
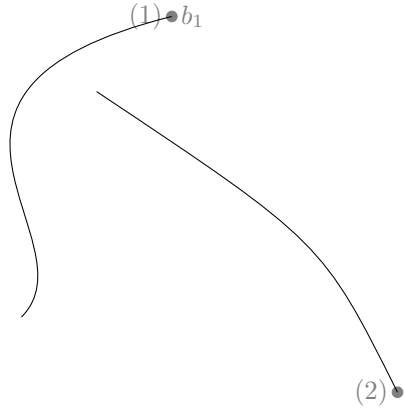
\begin{figure}[htb]
    \centering
\begin{tikzpicture}
\filldraw [gray] (-2,5) circle (0pt) node[anchor=west] {$X$};
\filldraw [gray] (5,-1) circle (2pt) node[anchor=east] {$(2)$};
\filldraw [gray] (2,4) circle (2pt) node[anchor=west] {$b_1$} node[anchor=east] {$(1)$};
\draw (0,0) .. controls (1,1) and (-2,3) .. (2,4);
\draw (1,3) .. controls (4,1) .. (5,-1);
 \end{tikzpicture}
 \caption{An algorithm for the $(2,1)$ robot motion planning problem}
    \label{fig2}
\end{figure}

The central problem of modern robotics, \textit{the motion planning problem}, consists of finding a motion planning algorithm.

We note that an algorithm to the $(k,r)$ robot motion planning problem is a (not necessarily continuous) section $s:F(X,k)\times F(X,r)\to PF(X,k)$ of the map $$e_{\pi_{k,r}^X}:PF(X,k)\to F(X,k)\times F(X,r),~e_{\pi_{k,r}^X}(\alpha)=(\alpha(0),\pi_{k,r}^X\alpha(1)),$$ where $\pi_{k,r}^X:F(X,k)\to F(X,r)$ is the projection of the first $r$ coordinates.

A motion planning algorithm $s$ is called \textit{continuous} if and only if $s$ is continuous. Absence of continuity will result in instability of the behavior of the motion planning. In general, there is not a global continuous motion planning algorithm, and only local continuous motion plans may be found. This fact gives, in a natural way, the use of the numerical invariant TC$(\pi_{k,r}^X)$. Recall that TC$(\pi_{k,r}^X)$ is the minimal number of \textit{continuous} local motion plans to $e_{\pi_{k,r}^X}$ (i.e., continuous local sections for $e_{\pi_{k,r}^X}$), which are needed to construct an algorithm for autonomous motion planning of the $(k,r)$ robot motion planning problem. Any motion  planning  algorithm $s:=\{s_i:U_i\to PE\}_{i=1}^{n}$ is called \textit{optimal} if $n=\text{TC}(\pi_{k,r}^X)$.

\begin{theorem}
Let $M$ be a connected topological manifold without boundary of dimension at least $2$, and let $\pi_{k,r}^X:F(M,k)\to F(M,r)$ be the Fadell-Neuwirth fibration. 
\begin{enumerate}
    \item If $M$ does not have the FPP, then $\text{TC}(\pi_{2,1}^M)=\text{TC}(M).$ Hence the complexity \red{of} the $(2,1)$ robot motion planning problem is the same \red{as} the complexity \red{of} the manifold $M$. More general, if $sec\hspace{.1mm}(\pi_{k,r}^M)=1$, then $\text{TC}(\pi_{k,r}^M)=\text{TC}(F(M,r)).$
    \item If $M$ has the FPP, then $\max\{2,\text{cat}(M)\}\leq \text{TC}(\pi_{k,1}^M)\leq\text{TC}(M)$, for any $k\geq 2$. In particular, $M$ is not contractible.
\end{enumerate}
\end{theorem}

\begin{example}
We recall that the $n$-dimensional sphere $S^n$ does not have the FPP. Then, $$\text{TC}(\pi_{2,1}^{S^n})=\text{TC}(S^n)=\begin{cases}
    2, & \hbox{for $n$ odd;} \\
    3, & \hbox{for $n$ even.}
\end{cases}
    $$
    Furthermore, we have that any contractible topological manifold $M$ without boundary does not have the FPP. Hence, $\text{TC}(\pi_{2,1}^M)=\text{TC}(M)=1$. 
\end{example}

\begin{example}
\begin{itemize}   
    \item The odd-dimensional projective spaces $\mathbb{RP}^m$ \red{do} not have the FPP, then $\text{TC}(\pi_{2,1}^{\mathbb{RP}^m})=\text{TC}(\mathbb{RP}^m)$. By \cite{farber2003topologicalproject}, the topological complexity $\text{TC}(\mathbb{RP}^m)$ for any $m\neq 1,3,7$, coincides with the smallest integer $k$ such that the projective space $\mathbb{RP}^m$ admits an immersion into $\mathbb{R}^{k-1}$.
    \item It is known that any closed surface $\Sigma$ except the projective plane $\Sigma\neq\mathbb{RP}^2$, does not have the FPP. Thus, $\text{TC}(\pi_{2,1}^{\Sigma})=\text{TC}(\Sigma).$ 
    \item We have that the projective plane $\mathbb{RP}^2$ has the FPP. Furthermore, it is well known $\text{cat}(\mathbb{RP}^2)=3$ and  $\text{TC}(\mathbb{RP}^2)=4$ \cite{farber2003topologicalproject}. Then, $3=\text{cat}(\mathbb{RP}^2)\leq\text{TC}(\pi_{k,1}^{\mathbb{RP}^2})\leq \text{TC}(\mathbb{RP}^2)=4$, for $k\geq 2$.
    \item For any connected compact Lie group, the Fadell-Neuwirth fibration $$\pi_{k,k-1}^{G\times \mathbb{R}^m}:F(G\times \mathbb{R}^m,k)\to F(G\times \mathbb{R}^m,k-1)$$ admits a continuous section (for $m\geq 2$). Then $\text{TC}(\pi_{k,k-1}^{G\times \mathbb{R}^m})=\text{TC}(F(G\times \mathbb{R}^m,k-1))$. By \cite{zapata2019cat}, the topological complexity $\text{TC}(F(G\times \mathbb{R}^m,2))=2\text{TC}(G)$. Hence, $\text{TC}(\pi_{3,2}^{G\times \mathbb{R}^m})=2\text{TC}(G)=2\text{cat}(G)$.
    \item Any connected Lie group has not the FPP and $\text{cat}(G)=\text{TC}(G)$. Then, $\text{TC}(\pi_{2,1}^G)=\text{TC}(G)=\text{cat}(G)$. In general, $\text{TC}(\pi_{k,1}^G)=\text{TC}(G)=\text{cat}(G)$ for any $k\geq 2$.
\end{itemize}
\end{example}

\begin{example}
\begin{itemize}
    \item We have $sec\hspace{.1mm}(\pi_{k,r}^{S^d})=cat(F(S^d,r))=2$, for $k\geq 3$, $d$ even, and $r=1,2$. Then $2=sec\hspace{.1mm}(\pi_{k,r}^{S^d})\leq\text{TC}(\pi_{k,r}^{S^d})\leq \text{TC}(F(S^d,r))=\text{TC}(S^d)=3$.
    \item For any $k\geq 2$ and $d$ odd, and $r=1,2$. We have $sec\hspace{.1mm}(\pi_{k,r}^{S^d})=1.$ Hence, $\text{TC}(\pi_{k,r}^{S^d})=\text{TC}(F(S^d,r))=\text{TC}(S^d)=2$.
\end{itemize}
\end{example}

\begin{proposition}\cite{pavesic2019}
Let $p:E\to B$ be a fibration between ANR spaces. Then \[\text{cat}(B)\leq\text{TC}(p)\leq\min\{ \text{cat}(E)+\text{cat}(E)sec\hspace{.1mm}(p)-1, \text{TC}(B),\text{cat}(E\times B)\}.\] In particular, $\text{TC}(p)=1$ if and only if $B$ is contractible.
\end{proposition}


\begin{theorem}
Let $M$ be a connected topological manifold without boundary of dimension at least $2$. 
If $M$ has the FPP, then $$\max\{2,\text{cat}(M)\}\leq\text{TC}(\pi_{2,1}^M)\leq \min\{ 3\text{cat}(F(M,2))-1, \text{TC}(M),\text{cat}(F(M,2)\times M)\}.$$
\end{theorem}




\bibliographystyle{plain}

\end{document}